\documentclass[a4paper,10pt]{article}
\usepackage{amsfonts}
\usepackage{bbm}
\usepackage{mathrsfs}
\usepackage{amsmath,amsthm,amssymb,amscd}
\usepackage[center]{titlesec}
\newtheorem{theo}{Theorem}[section]
\newtheorem{prop}[theo]{Proposition}
\newtheorem{lem}[theo]{Lemma}
\newtheorem{exa}[theo]{Example}
\newtheorem{rem}[theo]{Remark}

\newtheorem{defi}[theo]{Definition}

\newtheorem{ques}[theo]{Question}
\newcommand{\bp}{\begin{proof}}
\newcommand{\ep}{\end{proof}}

\begin{document}
 \setlength{\baselineskip}{13pt} \pagestyle{myheadings}

 \title{
 {Primitive Cohomology of Real Degree Two on Compact Symplectic Manifolds}
 \thanks{
 Supported by NSFC (China) Grants 11071208, 11371309 and the Postgraduate Innovation Project of Jiangsu Province (NO.CXZZ13$_{-}$0888).}
 }
 \author{{\large Qiang Tan, Hongyu Wang\thanks {E-mail:
 hywang@yzu.edu.cn }, Jiuru Zhou}\\
 }
 \date{}
 \maketitle

 \noindent {\bf Abstract:}
 In this paper, we define the generalized Lejmi's $P_J$ operator on a compact almost K\"{a}hler $2n$-manifold. We get that $J$ is $C^\infty$-pure and full if $\dim\ker P_J=b^2-1$.
 Additionally, we investigate the relationship between $J$-anti-invariant cohomology introduced by T.-J. Li and W. Zhang
 and new symplectic cohomologies introduced by L.-S. Tseng and S.-T. Yau on a closed symplectic $4$-manifold.
 \\

 \noindent {{\bf AMS Classification (2000):} 53C55, 53C22.}\\


 \section{Introduction}

 For an almost complex manifold $(M,J)$,
 T.-J. Li and W. Zhang \cite{LiZhC} introduced subgroups, $H_{J}^{+}$ and $H_{J}^{-}$, %
 of the real degree 2 de Rham cohomology group $H^2_{dR}(M,\mathbb R)$, as the sets of cohomology classes %
 which can be represented by $J$-invariant and $J$-anti-invariant real $2$-forms, respectively. %
 Let us denote by $h_J^{+}$ and $h_J^{-}$ the dimensions of $H_{J}^{+}$ and $H_{J}^{-}$, respectively.

 It is interesting to consider whether or not the subgroups $H_{J}^{+}$ and $H_{J}^{-}$ %
 induce a direct sum decomposition of $H^2_{dR}(M,\mathbb R)$.
 In the case of direct sum decomposition, $J$ is said to be $C^\infty$ pure and full.
 This is known to be true for integrable almost complex structures $J$
 which admit compatible K\"{a}hler metrics on compact manifolds of any dimension.
 In this case, the induced decomposition is nothing but the classical real
 Hodge-Dolbeault decomposition of $H^2_{dR}(M,\mathbb R)$ (see \cite{Barth}).

 In dimension 4, T. Draghici, T.-J. Li and W. Zhang \cite[Theorem 2.3]{DLZS}
 proved that on any closed almost complex $4$-manifold $(M,J)$, $J$ is $C^\infty$ pure and full.
 Further in \cite{DLZO}, they computed the subgroups $H_{J}^{+}$ and $H_{J}^{-}$
 and their dimensions $h_J^{+}$ and $h_J^{-}$ for almost complex structures metric related to an integrable one.

 \vskip 6pt

 In the fifth section of \cite{LejS1},
 Lejmi introduced the differential operator $P_J$ on a compact almost K\"{a}hler $4$-manifold $(M,g,J,\omega)$,
 \begin{eqnarray}
   P_J: \Omega^2_0 & \rightarrow & \Omega^2_0 \nonumber \\
        \psi & \mapsto & \frac{1}{2}\Delta_g\psi-\frac{1}{4}g(\Delta_g\psi,\omega)\omega. \nonumber
  \end{eqnarray}
 He proved that $P_J$ is a self-adjoint strongly elliptic linear operator of order $2$.
 In this paper, we define the generalized operator $P_J$ on a compact almost K\"{a}hler $2n$-manifold.
 We prove that $P_J$ is also a self-adjoint strongly elliptic linear operator on a compact almost K\"{a}hler manifold $(M,g,J,\omega)$
 of dimension $2n$.

 \vskip 6pt

 \noindent {\bf Proposition 2.3.}
  {\it Suppose that $(M,g,J,\omega)$ is a closed almost K\"{a}hler $2n$-manifold,
  then $\ker P_J=\mathcal{H}^-_J\oplus\mathcal{H}^+_{J,0}$ and the harmonic representatives of $H^-_J$ and $H^+_{J,0}$ are of pure degree,
  that is,
  $$H^-_J\cong\mathcal{H}^-_J,\,\,\, H^+_{J,0}\cong\mathcal{H}^+_{J,0}.$$
  }

 \vskip 6pt

 By studying the properties of $P_J$, we get that $J$ is $C^\infty$ pure and full when $\dim\ker P_J=b^2-1$.

 \vskip 6pt

 \noindent {\bf Theorem 2.5.}
 {\it Suppose that $(M,g,J,\omega)$ is a closed almost K\"{a}hler $2n$-manifold,
  if $\dim\ker P_J=b^2-1$, then $J$ is $C^\infty$ pure and full and
  \begin{eqnarray*}
     H^2_{dR}(M;\mathbb{R})&=&  H^+_J\oplus H^-_J \\
     &=&  Span_{\mathbb{R}}\{\omega\}\oplus H^+_{J,0}\oplus H^-_J\\
     &=&  H^{(1,0)}_\omega(M;\mathbb{R})\oplus H^{(0,2)}_\omega(M;\mathbb{R}).
  \end{eqnarray*}
  Moreover, $J$ is pure and full.}

 \vskip 6pt

 Recently, L.-S. Tseng and S.-T. Yau \cite{TY} introduced new cohomologies for a closed symplectic manifold $M$.
 On a compact symplectic manifold $(M,\omega)$ of dimension $2n$, the symplectic star operator $*_s$ acts on a differential
 $k$-form $\alpha$ by
 \begin{eqnarray*}
 \alpha\wedge*_s\alpha' &=&(\omega^{-1})^k(\alpha,\alpha')d{\rm vol} \\
 &=& \frac{1}{k!}(\omega^{-1})^{i_1j_1}\cdot\cdot\cdot(\omega^{-1})^{i_kj_k}\alpha_{i_1\cdot\cdot\cdot i_k}\alpha'_{j_1\cdot\cdot\cdot j_k}\frac{\omega^n}{n!}
 \end{eqnarray*}
 with repeated indices summed over.
 Note that $*_s*_s=1$, which follows from Weil's identity \cite{GuiS,WeiI}
  \begin{equation*}
     *_s\frac{L^r}{r!}B_k=(-1)^{k(k+1)/2}\frac{L^{n-r-k}}{(n-r-k)!}B_k
  \end{equation*}
  for any primitive $k$-form $B_k$.
  Also by Weil's identity, for any primitive $k$-form $B_k$, we can get
  \begin{equation}\label{Weil's relation}
     *_g\frac{L^r}{r!}B_k=(-1)^{k(k+1)/2}\frac{L^{n-r-k}}{(n-r-k)!}\mathcal{J}(B_k),
  \end{equation}
  where $\mathcal{J}=\sum_{p,q}(\sqrt{-1})^{p-q}\Pi^{p,q}$
  projects a $k$-form onto its $(p,q)$ parts times the multiplicative factor $(\sqrt{-1})^{p-q}$.
 The adjoint of the standard exterior derivative takes the form
 $$
 d^\Lambda=(-1)^{k+1}*_sd*_s.
 $$
 By using the properties $d^2=(d^\Lambda)^2=0$ and the anti-commutively $dd^\Lambda=-d^\Lambda d$,
 Tseng and Yau \cite{TY} considered new symplectic cohomology groups $H^k_{d+ d^\Lambda}(M)$ and $H^k_{dd^\Lambda}(M)$.
 They also proved that the space of $d+d^\Lambda$-harmonic $k$-forms $\mathcal{H}^k_{d+d^\Lambda}(M)$ and
 the space of $dd^\Lambda$-harmonic $k$-forms $\mathcal{H}^k_{dd^\Lambda}(M)$ are finite dimensional and
 isomorphic to $H^k_{d+d^\Lambda}(M)$ and $H^k_{dd^\Lambda}(M)$, respectively.

 By considering the relationship between $H^-_J(\cong\mathcal{H}_J^-)$ and symplectic cohomology groups on a closed almost K\"{a}hler $4$-manifold,
 we obtain the following theorem.

 \vskip 6pt

 \noindent {\bf Theorem 3.2.}
   {\it Suppose that $(M,g,J,\omega)$ is a closed almost K\"{a}hler $2n$-manifold, then
   $$\ker P_J=\mathcal{H}^-_{d+d^\Lambda}(M)\cap \mathcal{H}^-_{dd^\Lambda}(M).$$
   If $\mathcal{H}^-_{d+d^\Lambda}(M)=\mathcal{H}^-_{dd^\Lambda}(M)$, then
   $$\mathcal{H}^2_{d+d^\Lambda}(M)=\mathcal{H}^2_{dd^\Lambda}(M)=Span_{\mathbb{R}}\{\omega\}\oplus \mathcal{H}^-_J\oplus\mathcal{H}^+_{J,0}.$$
   In particular, if $n=2$, $$\mathcal{H}^-_{d+d^\Lambda}(M)=\mathcal{H}^-_J\oplus\mathcal{H}_g^-\oplus(\mathcal{H}^-_J\oplus\mathcal{H}_g^-)^{-,\perp}_{d+d^\Lambda},$$
   $$\mathcal{H}^-_{dd^\Lambda}(M)=\mathcal{H}^-_J\oplus\mathcal{H}_g^-\oplus(\mathcal{H}^-_J\oplus\mathcal{H}_g^-)^{-,\perp}_{dd^\Lambda},$$
   $$*_g(\mathcal{H}^-_J\oplus\mathcal{H}_g^-)^{-,\perp}_{d+d^\Lambda}=(\mathcal{H}^-_J\oplus\mathcal{H}_g^-)^{-,\perp}_{dd^\Lambda}.$$
   }

 \section{Primitive de Rham cohomology of degree two}\setcounter{equation}{0}

  An almost K\"{a}hler structure on a real manifold $M$ of dimension $2n$ is given by a triple $(g,J,\omega)$ of
  a Riemannian metric $g$, an almost complex structure $J$ and a symplectic form $\omega$, which satisfies the compatibility
  relation
  \begin{equation}\label{compatible metric}
    g(\cdot,\cdot)=\omega(\cdot,J\cdot).
  \end{equation}
  We say that the almost complex structure $J$ is $\omega$ compatible if it induces a Riemannian metric via (\ref{compatible metric}).

  Suppose that $(M,g,J,\omega)$ is a closed almost K\"{a}hler $2n$-manifold.
  The almost complex structure $J$ acts on the space $\Omega^2$ of smooth $2$-forms on $M$ as an involution by
  \begin{equation}\label{involution}
    \alpha \longmapsto \alpha(J\cdot,J\cdot), \quad \alpha\in\Omega^2(M).
  \end{equation}
  This gives the $J$-invariant, $J$-anti-invariant decomposition of $2$-forms (see \cite{DonT}):
  \begin{equation*}
  \Omega^2 = \Omega^+_J \oplus \Omega^-_J, \quad \alpha = \alpha_J^+ + \alpha_J^-
  \end{equation*}
  as well as the splitting of corresponding vector bundles
  \begin{equation}\label{J decomposition}
  {\Lambda}^2={\Lambda}_J^+ \oplus {\Lambda}_J^-.
  \end{equation}
  Let $\mathcal Z^2$ denote the space of closed $2$-forms on $M$ and set
  \begin{align*}
  \mathcal Z_J^+ \triangleq \mathcal Z^2 \cap \Omega_J^+, \quad
  \mathcal Z_J^- \triangleq \mathcal Z^2 \cap \Omega_J^-.
  \end{align*}
  Define the $J$-invariant and $J$-anti-invariant cohomology subgroups $H_J^{\pm}$ (see \cite{LiZhC}) by
  \begin{align*}
  H^\pm_J=\{\mathfrak a \in H^2_{dR}(M;\mathbb{R}) \mid
           \mbox{there exists} \ \alpha \in \mathcal Z_J^{\pm} \ \mbox{such that} \ \mathfrak a = [\alpha]\}.
  \end{align*}
  Let us denote by $h_J^{+}$ and $h_J^{-}$ the dimensions of $H_{J}^{+}$ and $H_{J}^{-}$, respectively.
  We say $J$ is $C^\infty$ pure if $H^+_J\cap H^-_J=\{0\}$, $C^\infty$ full if $H^+_J+H^-_J=H^2_{dR}(M,\mathbb{R})$,
  and $J$ is $C^\infty$ pure and full if
  $$H^2_{dR}(M,\mathbb{R})=H^+_J\oplus H^-_J.$$
  T. Draghici, T.-J. Li and W. Zhang have proved that for any closed almost complex $4$-manifold $(M,J)$,
  $J$ is $C^\infty$ pure and full (see \cite{DLZS}).

  On a smooth closed manifold $M$, the space $\Omega^*(M)$ of smooth forms is a vector space,
  and with $C^\infty$ topology, it is a Fr\'{e}chet space.
  The space $\mathcal{E}_*(M)$ of currents is the topological dual space, which is also a Fr\'{e}chet space (see \cite{GH,LinS}).
  As a topological vector space, $\Omega^*(M)$ is reflexive, thus it is also the dual space of $\mathcal{E}_*(M)$.
  Denote the space of closed currents by $\mathcal{Z}_*(M)$ and the space of boundaries by $\mathcal{B}_*(M)$.
  On a closed almost complex manifold $(M,J)$, there is a natural action of $J$ on the space $\Omega^k(M)_{\mathbb{C}}\triangleq\Omega^k(M)\otimes\mathbb{C}$, which induces a topological type decomoposition
  $$\Omega^k(M)_{\mathbb{C}}=\bigoplus_{p+q=k}\Omega^{p,q}_J(M)_{\mathbb{C}}.$$
  If $k$ is even, $J$ also acts on $\Omega^k(M)$ as an involution.
  Specifically, if $k=2$, $J$ acts on $\Omega^2(M)$ as (\ref{involution}) and decomposes it into the topological direct sum of the invariant part $\Omega^+_J$ and the anti-invariant part $\Omega^-_J$.
  In this case, the two decompositions are related in the following way:
  $$\Omega^+_J(M)=\Omega^{1,1}_J(M)_{\mathbb{R}}\triangleq\Omega^{1,1}_J(M)_{\mathbb{C}}\cap\Omega^2(M),$$
  $$\Omega^-_J(M)=\Omega^{(2,0),(0,2)}_J(M)_{\mathbb{R}}\triangleq(\Omega^{(2,0)}_J(M)_{\mathbb{C}}\oplus\Omega^{(2,0)}_J(M)_{\mathbb{C}})
    \cap\Omega^2(M).$$
  For the space of real $2$-currents, we have a similar decomposition
  $$\mathcal{E}_2(M)=\mathcal{E}^J_{1,1}(M)_{\mathbb{R}}\oplus\mathcal{E}^J_{(2,0),(0,2)}(M)_{\mathbb{R}},$$
  and the corresponding subspaces of closed and boundary currents,
  $$\mathcal{B}^J_{1,1}\subset\mathcal{Z}^J_{1,1}\subset\mathcal{E}^J_{1,1}(M)_{\mathbb{R}},$$
  $$\mathcal{B}^J_{(2,0),(0,2)}\subset\mathcal{Z}^J_{(2,0),(0,2)}\subset\mathcal{E}^J_{(2,0),(0,2)}(M)_{\mathbb{R}}.$$
  We note the dual space of $\mathcal{E}^J_{1,1}(M)_{\mathbb{R}}$ is $\Omega^{1,1}_J(M)_{\mathbb{R}}$, and vice versa.
  Similarly, $\mathcal{E}^J_{(2,0),(0,2)}(M)_{\mathbb{R}}$ is the dual space of $\Omega^{(2,0),(0,2)}_J(M)_{\mathbb{R}}$.
  If $S=(1,1)$ or $(2,0),(0,2)$ (cf.{\cite{FTO,LiZhC}}), define
  $$H^J_S(M)_\mathbb{R}=\frac{\mathcal{Z}^J_S}{\mathcal{B}^J_S}.$$
  $J$ is said to be pure if
  $$\frac{\mathcal{Z}^J_{1,1}}{\mathcal{B}^J_{1,1}}\cap\frac{\mathcal{Z}^J_{(2,0),(0,2)}}{\mathcal{B}^J_{(2,0),(0,2)}}=0.$$
  $J$ is said to be full if
  $$\frac{\mathcal{Z}_2}{\mathcal{B}_2}=\frac{\mathcal{Z}^J_{1,1}}{\mathcal{B}^J_{1,1}}+
  \frac{\mathcal{Z}^J_{(2,0),(0,2)}}{\mathcal{B}^J_{(2,0),(0,2)}}.$$
  Therefore, an almost complex structure $J$ is pure and full if and only if
  \begin{equation}\label{pure full}
   H_2(M;\mathbb{R})=H^J_{1,1}(M)_\mathbb{R}\oplus H^J_{(2,0),(0,2)}(M)_\mathbb{R},
  \end{equation}
  where $H_2(M;\mathbb{R})$ is the $2$-nd de Rham homology group.

  In particular, if $(M,g,J,\omega)$ is a closed almost K\"{a}hler $4$-manifold,
  then the Hodge star operator $*_g$ gives
  the well-known self-dual, anti-self-dual decomposition of $2$-forms:
  \begin{equation*}
  \Omega^2 = \Omega_g^+ \oplus \Omega_g^-
  \end{equation*}
  as well as the corresponding splitting of the bundles (see \cite{DonT, DKT})
  \begin{equation}\label{g decomposition}
  {\Lambda}^2 = {\Lambda}_g^+ \oplus {\Lambda}_g^-.
  \end{equation}
  The Hodge-de Rham Laplacian commutes with $*_g$,
  so the decomposition (\ref{g decomposition}) holds for the space $\mathcal {H}^2_g$ of harmonic $2$-forms as well.
  By Hodge theory, this induces cohomology decomposition by the metric $g$:
  \begin{align}\label{g harmonic decomposition}
  H^2_{dR}(M;\mathbb{R})\cong\mathcal {H}^2_g=\mathcal {H}_g^+\oplus\mathcal
  {H}_g^-.
  \end{align}
  One defines (see \cite{DKT})
  \begin{align*}
  H_g^{\pm} = \{ \mathfrak a \in H^2_{dR}(M;\mathbb{R}) \mid
             \mathfrak a = [\alpha] \,\; \mbox{for some} \,\;\alpha \in \mathcal Z_g^{\pm}:=\mathcal Z^2\cap\Omega_g^{\pm} \}. %
  \end{align*}
  It is easy to see that
  \begin{align*}
  H_g^\pm \cong \mathcal Z_g^{\pm} = \mathcal {H}_g^\pm
  \end{align*}
  and (\ref{g harmonic decomposition}) can be written as
  \begin{align*}
  H^2_{dR}(M;\mathbb{R}) = H_g^+ \oplus H_g^-.
  \end{align*}

  There are the following relations between the decompositions \eqref{J decomposition} and \eqref{g decomposition}
  on an almost K\"{a}hler $4$-manifold (cf. \cite{DonT,DKT}):
  \begin{align*}
  & {\Lambda}_J^+=\langle\omega\rangle\oplus{\Lambda}_g^-,  \\
  & {\Lambda}_g^+=\langle\omega\rangle\oplus{\Lambda}_J^-, \\
  & {\Lambda}_J^+\cap{\Lambda}_g^+=\langle\omega\rangle, \ \ \ \ {\Lambda}_J^-\cap{\Lambda}_g^-=\{0\}.
  \end{align*}
  It is easy to see that
  $\mathcal{Z}_J^-\subset \mathcal {H}_g^+$ and $H_g^-\subset H_J^+$.
  Let $b^2$, $b^{+}$ and $b^{-}$  be the second, the self-dual and the anti-self-dual Betti number of $M$, respectively. Thus $b^2=b^{+}+b^{-}$.
  It is easy to see that, for a closed almost K\"{a}hler $4$-manifold $(M,g,J,\omega)$, there hold (see \cite{DLZS, DLZO,TWZZ}):
  \begin{align}\label{dimension}
  H_J^- \cong \mathcal Z_J^-, \quad
  h_J^+ + h_J^- = b^2, \quad
  h_J^+ \geq b^-+1, \quad 0\leq h_J^- \leq b^+-1.
  \end{align}

  \vskip 6pt

  Suppose that $(M,g,J,\omega)$ is a closed almost K\"{a}hler $2n$-manifold.
  A differential $k$-form $B_k$ with $k\leq n$ is called primitive if $L^{n-k+1}B_k=0$ (see \cite{TY,YanH}).
  Here $L$ is the Lefschetz operator (see \cite{BrA,TY,YanH}) which is defined acting on a $k$-form $A_k\in\Omega^k(M)$ by
  $$L(A_k)=\omega\wedge A_k.$$
  Define the space of primitive $k$-forms by $\Omega^k_0$.
  Specifically,
  \begin{equation*}
    \Omega^2_0=\{\alpha\in\Omega^2:\omega^{n-1}\wedge\alpha=0\}.
  \end{equation*}
  Therefore,
  \begin{equation*}
    \Omega^2=\Omega^2_1\oplus\Omega^2_0,
  \end{equation*}
  where $\Omega^2_1\triangleq\{f\omega: f\in C^\infty(M) \}$.
  It is easy to see that $\Omega^-_J\subset\Omega^2_0$.
  So we can get the following decomposition
  \begin{equation}
    \Omega^2_0=\Omega^+_{J,0}\oplus\Omega^-_J,
  \end{equation}
  where $\Omega^+_{J,0}$ is the space of the primitive $J$-invariant $2$-forms.
  We consider the following second order linear differential operator on $\Omega^2_0$.
  \begin{eqnarray}
   P_J: \Omega^2_0 & \rightarrow & \Omega^2_0 \nonumber \\
        \psi & \mapsto & \Delta_g\psi-\frac{1}{n}g(\Delta_g\psi,\omega)\omega, \nonumber
  \end{eqnarray}
  where $\Delta_g$ is the Riemannian Laplacian with respect to the metric $g(\cdot,\cdot)=\omega(\cdot,J\cdot)$ (here we use the convention $g(\omega,\omega)=n$).
  \begin{lem}\label{P lemma}
  $P_J$ is a self-adjoint strongly elliptic linear operator with kernel the primitive $g$-harmonic $2$-forms.
  \end{lem}
  \begin{proof}
  We claim that
  \begin{eqnarray*}
   (P_J-\Delta_g)(\psi)&=& -\frac{1}{n}[d^*(\nabla^g\psi,\omega)_g-(\nabla^g\psi,\nabla^g\omega)_g-(2-\frac{4}{n-2})Tr(\omega\cdot Ric\cdot \psi) \\
                     &~& -\frac{2s^g}{(n-1)(n-2)}(\psi,\omega)_g-W^g(\psi,\omega)]\omega \\
                     &=& -\frac{1}{n}[-d^*(\psi,\nabla^g\omega)_g-(\nabla^g\psi,\nabla^g\omega)_g-(2-\frac{4}{n-2})Tr(\omega\cdot Ric\cdot \psi) \\
                     &~& -\frac{2s^g}{(n-1)(n-2)}(\psi,\omega)_g-W^g(\psi,\omega)]\omega ,
  \end{eqnarray*}
  where $W^g$ is the Weyl tensor (see \cite{BesE}), $\nabla^g$ is the Levi-Civita connection, $s^g$ is the Riemannian scalar curvature
  with respect to the metric $g$ and $Tr(\omega\cdot Ric\cdot \psi)=\Sigma_{i,j,k}\omega_{ij}R_{jk}\psi_{ik}$
  is the trace of $\omega\otimes Ric\otimes \psi$.
  Indeed, by Weitzenb\"{o}ck-Bochner formula (see \cite{BesE}), we will get
  \begin{eqnarray}\label{WB formula}
    (\Delta_g\psi-Tr(\nabla^g)^2\psi,\omega)_g &=& (\Delta_g\psi-d^*(\nabla^g)\psi,\omega)_g \nonumber\\
    &=& \sum_{i,j,k,l}2R_{iklj}\psi_{kl}\omega_{ij}-\sum_{i,j,k,l}R_{ik}\psi_{kj}\omega_{ij}-\sum_{i,j,k,l}R_{jk}\psi_{ik}\omega_{ij} \nonumber\\
    &=& \sum_{i,j,k,l}2R_{iklj}\psi_{kl}\omega_{ij}-\sum_{i,j,k,l}R_{ik}\psi_{kj}\omega_{ij}-\sum_{i,j,k,l}R_{ik}\psi_{jk}\omega_{ji} \nonumber\\
    &=& \sum_{i,j,k,l}2R_{iklj}\psi_{kl}\omega_{ij}-\sum_{i,j,k,l}2R_{ik}\psi_{kj}\omega_{ij}.
  \end{eqnarray}
  On the other hand,
  \begin{eqnarray}\label{Weyl}
   W^g(\omega,\psi) &=& (R-\frac{1}{n-2}Ric\circ g+\frac{s^g}{2(n-1)(n-2)}g\circ g)(\omega,\psi) \nonumber\\
                    &=& \sum_{i,j,k,l}\omega_{ij}\psi_{kl}R_{ijkl}-\sum_{i,j,k,l}\frac{1}{n-2}\omega_{ij}\psi_{kl}(R_{il}g_{jk}
                        +R_{jk}g_{il} \nonumber\\
                    &~& -R_{ik}g_{jl}-R_{jl}g_{ik}) +\frac{s^g}{2(n-1)(n-2) }g\circ g(\omega,\psi) \nonumber\\
                    &=& \sum_{i,j,k,l}\omega_{ij}\psi_{kl}R_{ijkl}-\sum_{i,k,l}\frac{4}{n-2}\omega_{ik}\psi_{kl}R_{il}
                        -\sum_{i,j}\frac{2s^g}{(n-1)(n-2)}\omega_{ij}\psi_{ij} \nonumber\\
                    &=&  \sum_{i,j,k,l}\omega_{ij}\psi_{kl}(-R_{iljk}-R_{iklj})-
                        \sum_{i,k,l}\frac{4}{n-2}\omega_{ik}\psi_{kl}R_{il} \nonumber\\
                    &~&  -\sum_{i,j}\frac{2s^g}{2(n-1)(n-2)}\omega_{ij}\psi_{ij} \nonumber\\
                    &=& \sum_{i,j,k,l}-2R_{iklj}\psi_{kl}\omega_{ij}-
                        \sum_{i,k,l}\frac{4}{n-2}R_{il}\psi_{kl}\omega_{ik} \nonumber\\
                    &~& -\sum_{i,j}\frac{2s^g}{2(n-1)(n-2)}\omega_{ij}\psi_{ij}.
   \end{eqnarray}
  Here we compute (\ref{WB formula}) and (\ref{Weyl}) under the local coordinates system $(x^1,x^2,\cdot\cdot\cdot,x^{2n})$.
  In addtion, we suppose $(\omega_{ij})$ to be the local representation of $\omega$.
  Similarly,
  $\psi=(\psi_{ij})$, $Ric=(R_{ij})$ and $R=(R_{ijkl})$, where $Ric$ is the Ricci curvature tensor and $R$ is the Riemannian curvature tensor with respect to metric $g$ .
  By (\ref{WB formula}) and (\ref{Weyl}), we can get the above claim.
  It is easy to see that $P_J-\Delta_g$ is a linear differential operator of order $1$.
  So the operator $P_J$ is a self-adjoint strongly elliptic linear operator of order $2$.

  It remains to prove that the kernel of $P_J$ is the space of the primitive $g$-harmonic $2$-forms,
  that is, $\mathcal{H}^2_g\cap\Omega^2_0$.
  Clearly, $\mathcal{H}^2_g\cap\Omega^2_0\subset \ker P_J$.
  For any $\psi\in\ker P_J$,
  \begin{eqnarray*}
    0 &=& \int_M(P_J(\psi),\psi)_gdvol_g \\
      &=& \int_M(\Delta_g\psi-\frac{1}{n}(\Delta_g\psi,\omega)_g\omega,\psi)_g \\
      &=& \int_M(\Delta_g\psi,\psi)_g \\
      &=& \int_M(d\psi,d\psi)_g+(d^*\psi,d^*\psi)_g.
  \end{eqnarray*}
  Hence, $d\psi=d^*\psi=0$ and $\psi$ is a primitive $g$-harmonic $2$-form.
  So we get $\ker P_J=\mathcal{H}^2_g\cap\Omega^2_0$.
  \end{proof}

  \begin{rem}{\rm
  Let $J_t$, $t\in[0,1]$ be a smooth family of $\omega$-compatible almost complex structures
  on $M$, then $\dim\ker P_{J_t}$ is an upper-semi-continuous function in $t$,
  by a classical result of Kodaira and Morrow showing the upper-semi-continuity of the kernel of a family of elliptic
  differential operators (Theorem 4.3 in \cite{KMC}).
  }
  \end{rem}

  We define $\mathcal{H}^-_J$ to be the space of the harmonic $J$-anti-invariant 2-forms and
  $\mathcal{H}^+_{J,0}$ to be the space of the harmonic primitive $J$-invariant 2-forms.
  Define the primitive $J$-invariant cohomology subgroup $H^+_{J,0}$ by
  \begin{align*}
  H^+_{J,0}=\{\mathfrak a \in H^2_{dR}(M;\mathbb{R}) \mid
           \mbox{there exists} \ \alpha \in \mathcal Z^2\cap\Omega^+_{J,0} \ \mbox{such that} \ \mathfrak a = [\alpha]\}.
  \end{align*}

  \begin{prop}\label{col kerP}
  Suppose that $(M,g,J,\omega)$ is a closed almost K\"{a}hler $2n$-manifold,
  then $\ker P_J=\mathcal{H}^-_J\oplus\mathcal{H}^+_{J,0}$ and the harmonic representatives of $H^-_J$ and $H^+_{J,0}$ are of pure degree,
  that is,
  $$H^-_J\cong\mathcal{H}^-_J,\,\,\, H^+_{J,0}\cong\mathcal{H}^+_{J,0}.$$
  \end{prop}
  \begin{proof}
   For any $\alpha\in\ker P_J$, $\alpha$ is primitive and harmonic.
  So $d\alpha=0, d*_g\alpha=0$
  and $\alpha$ can be written as
  $\alpha=\beta_{\alpha}+\gamma_{\alpha}$, where $\beta_{\alpha}\in\Omega^-_J$ and $\gamma_{\alpha}\in\Omega^+_{J,0}$.
  By a direct computation, we get
  $$*_g\alpha=\frac{L^{n-2}}{(n-2)!}(\beta_{\alpha}-\gamma_{\alpha})$$
  and $$(\frac{L^{n-2}}{(n-2)!}+*_g)\alpha=2\frac{L^{n-2}}{(n-2)!}\beta_{\alpha}.$$
  Hence,
  $$d(2\frac{L^{n-2}}{(n-2)!}\beta_{\alpha})=d(\frac{L^{n-2}}{(n-2)!}+*_g)\alpha=0.$$
  So we can get $d\beta_{\alpha}=0$.
  Since $*_g\beta_{\alpha}=\frac{L^{n-2}}{(n-2)!}\beta_{\alpha}$, $d*_g\beta_{\alpha}=0$.
  Therefore, $\beta_{\alpha}\in\mathcal{H}^-_J$.
  Similarly,
  $$(\frac{L^{n-2}}{(n-2)!}-*_g)\alpha=2\frac{L^{n-2}}{(n-2)!}\gamma_{\alpha}$$
  and we can get $\gamma_{\alpha}\in\mathcal{H}^+_{J,0}$.
  Thus, $\ker P_J=\mathcal{H}^-_J\oplus\mathcal{H}^+_{J,0}$.

  For any $\mathfrak a=[\alpha]\in H^-_J$, $\alpha\in\mathcal{Z}^-_J$.
  By (\ref{Weil's relation}),
  $$*_g\alpha=-\frac{L^{n-2}}{(n-2)!}\mathcal{J(\alpha)}=\frac{L^{n-2}}{(n-2)!}\alpha.$$
  So $d*_g\alpha=d\frac{L^{n-2}}{(n-2)!}\alpha=0$ and $\alpha$ is a harmonic $J$-anti-invariant form, that is, $\alpha\in\mathcal{H}^-_J$.
  Hence, $H^-_J\cong\mathcal{H}^-_J$.
  Similarly, $H^+_{J,0}\cong\mathcal{H}^+_{J,0}$.
  The harmonic representatives of $H^-_J$ and $H^+_{J,0}$ are of pure degree.
  \end{proof}

  \begin{rem}{\rm
  In case $n=2$, on a closed almost K\"{a}hler $4$-manifold, Lejmi \cite{LejS1} proved that $P_J$ preserves the decomposition
  $$\Omega^2_0=\Omega^+_{J,0}\oplus\Omega^-_J.$$
  Furthermore,
  $P_J|_{\Omega^+_{J,0}}(\psi)=\Delta_g\psi$ and
  $P_J|_{\Omega^-_J}(\psi)=2d^-_Jd^*\psi$.
  He also pointed out that $P_J|_{\Omega^-_J}(\psi)=2d^-_Jd^*\psi$ is a self-adjoint strongly elliptic linear operator from
  $\Omega^-_J$ to $\Omega^-_J$ on a closed almost K\"{a}hler $4$-manifold.
  It follows that the kernel of $P_J$ consists of primitive harmonic $2$-forms which splits
  as anti-self-dual and $J$-anti-invariant ones. So he gets
  $$\dim\ker P_J=b^-+h^-_J.$$
  But when $n>2$, by computing the principal symbol of $d^-_Jd^*$, one finds that $d^-_Jd^*$ is no longer a self-adjoint strongly elliptic linear operator.
  So we are not able to get any good properties about the $\dim\ker P_J$ in higher dimension.
  }
  \end{rem}
  In \cite{Angella}, D. Angella and A. Tomassini define
  $$H^{(r,s)}_\omega(M;R)\triangleq\{[L^r\beta]\in H^{2r+s}_{dR}(M;\mathbb{R}): \beta\in \Omega^s_0\}\subseteq H^{2r+s}_{dR}(M;\mathbb{R})$$
  for $r,s\in \mathbb{N}$.
  Obviously, for every $k\in \mathbb{N}$, one has
  $$\sum_{2r+s=k}H^{(r,s)}_\omega(M;\mathbb{R})\subseteq H^{2r+s}_{dR}(M;\mathbb{R}).$$
  A natural question is that when the above inclusion is actually an equality, and when the sum is a direct sum.
  Fortunately, Angella and Tomassini have proved that
  \begin{equation}\label{AT}
    H^2_{d\mathbb{R}}(M;\mathbb{R})=H^{(1,0)}_\omega(M;\mathbb{R})\oplus H^{(0,2)}_\omega(M;\mathbb{R})
  \end{equation}
  in \cite{Angella}.
  Clearly,  $H^{(1,0)}_\omega(M;\mathbb{R})\cong Span_{\mathbb{R}}\{\omega\}$
  and $\dim H^{(0,2)}_\omega(M;\mathbb{R})=b^2-1$.
  Here we want to emphasize that $\ker P_J\cong H^-_J\oplus H^+_{J,0}\subseteq H^{(0,2)}_\omega(M;\mathbb{R})$
  and if $\dim\ker P_J=b^2-1$, then $\ker P_J\cong H^-_J\oplus H^+_{J,0}=H^{(0,2)}_\omega(M;\mathbb{R})$.

  \vskip 6pt

  It is well known that on any closed almost complex $4$-manifold $(M,J)$, $J$ is $C^\infty$ pure and full (see \cite{DLZS}).
  But we can not get this result in higher dimension.
  In \cite{FTO}, A. Fino and A. Tomassini showed the existence of a compact $6$-dimensional nil-manifold with an almost complex structure
  which is not $C^\infty$ pure, i.e., the intersection of $H^+_J(M)$ and $H^-_J(M)$ is non-empty.
  They also prove that on a compact almost complex $2n$-manifold $(M,J)$, if $J$ admits a compatible symplectic structure, then $J$ is $C^\infty$ pure. With this result and by studying the dimension of $\dim\ker P_J$, we can get the following theorem.
  \begin{theo}
  Suppose that $(M,g,J,\omega)$ is a closed almost K\"{a}hler $2n$-manifold,
  if $\dim\ker P_J=b^2-1$, then $J$ is $C^\infty$ pure and full and
  \begin{eqnarray*}
     H^2_{dR}(M;\mathbb{R})&=&  H^+_J\oplus H^-_J \\
     &=&  Span_{\mathbb{R}}\{\omega\}\oplus H^+_{J,0}\oplus H^-_J\\
     &=&  H^{(1,0)}_\omega(M;\mathbb{R})\oplus H^{(0,2)}_\omega(M;\mathbb{R}).
  \end{eqnarray*}
  Moreover, $J$ is pure and full.
  \end{theo}
  \begin{proof}

  In \cite{FTO}, A. Fino and A. Tomassini showed that $J$ is $C^\infty$ pure if $(M,g,J,\omega)$ is a
  closed almost K\"{a}hler $2n$-manifold.
  In \cite{DLZS}, T. Draghici, T.-J. Li and W. Zhang proved the same result on a closed almost K\"{a}hler $2n$-manifold.
  Next, we will prove that $J$ is $C^\infty$ full under the condition
  of $\dim\ker P_J=b^2-1$.

  By proposition \ref{col kerP},
  if $\dim\ker P_J=b^2-1$, we get
  $$
  \dim \mathcal{H}^-_J+\dim \mathcal{H}^+_{J,0}=b^2-1.
  $$
  Hence,
  \begin{equation}\label{b^2 decompose}
    b^2=\dim \mathcal{H}^-_J+(\dim \mathcal{H}^+_{J,0}+1).
  \end{equation}
  It is easy to see that
  $$H^+_J\oplus H^-_J \subseteq H^2_{dR}(M;\mathbb{R}),$$
  and
  $$Span_{\mathbb{R}}\{\omega\}\oplus H^+_{J,0}\subseteq H^+_J.$$
  So we will get
  \begin{equation}\label{b^2}
    h^+_J+h^-_J\leq b^2
  \end{equation}
  and
  \begin{equation}\label{h^+_J}
  \dim H^+_{J,0}+1=\dim \mathcal{H}^+_{J,0}+1\leq h^+_J.
  \end{equation}
  Therefore, by (\ref{b^2 decompose}), (\ref{b^2}) and (\ref{h^+_J}),
  $$h^+_J+h^-_J\leq b^2=\dim \mathcal{H}^-_J+(\dim \mathcal{H}^+_{J,0}+1)\leq h^+_J+h^-_J.$$
  Clearly, $h^+_J+h^-_J=b^2$ and $\dim \mathcal{H}^+_{J,0}+1=h^+_J$.
  So we can get
  $$H^2_{dR}(M;\mathbb{R})=H^+_J\oplus H^-_J$$
  and
  $$Span_{\mathbb{R}}\{\omega\}\oplus H^+_{J,0}= H^+_J.$$
  Then we can get the following decompositions
  \begin{eqnarray*}
     H^2_{dR}(M;\mathbb{R})&=&  H^+_J\oplus H^-_J \\
     &=&  Span_{\mathbb{R}}\{\omega\}\oplus H^+_{J,0}\oplus H^-_J\\
     &=&  H^{(1,0)}_\omega(M;\mathbb{R})\oplus H^{(0,2)}_\omega(M;\mathbb{R})
  \end{eqnarray*}
  and
  $H^+_J\cong\mathcal{H}^+_J$,
  where $\mathcal{H}^+_J$ is the space of the harmonic J-invariant 2-forms.
  Of course, $J$ is $C^\infty$ pure and full.

  We have proven that $J$ is $C^\infty$ pure and full.
  Then by Theorem $3.7$ in \cite{FTO}, we can get that $J$ is pure.
  In addition, in the above proof, we have gotten that
  $H^-_J\cong\mathcal{H}^-_J$ and  $H^+_J\cong\mathcal{H}^+_J$ when $\dim\ker P_J=b^2-1$.
  Hence the harmonic representatives of $H^2_{dR}(M;\mathbb{R})$ are of pure degree (cf. \cite{Angella2}).
  Also by Theorem $3.7$ in \cite{FTO}, we can get that $J$ is pure and full.
  This completes the proof of the Theorem.
  \end{proof}

  In the above theorem, $\dim\ker P_J=b^2-1$ is just the sufficient condition but not the necessary condition
  for $J$'s $C^\infty$ pureness and fullness.
  Just by $J$'s $C^\infty$ pureness and fullness, we can not get $\dim\ker P_J=b^2-1$.
  Indeed, we have the following counter-example which is constructed by T. Draghici  (\cite{Dra}).

  \begin{exa}\label{example1}
  {\rm (also cf. \cite{Angella2}) Let $\mathbb{T}^4$ be the standard torus with coordinates $\{x^1,x^2,x^3,x^4\}$.
  Denote by $(g_0,J_0,\omega_0)$ be a standard flat K\"{a}hler structure on $\mathbb{T}^4$,
  so $(\mathbb{T}^4, \omega_0)$ has the hard Lefschetz property (cf. \cite{YanH}),
  that is, the map
 \[
   H^k_{dR}(\mathbb{T}^4;\mathbb{R})\rightarrow H^{4-k}_{dR}(\mathbb{T}^4;\mathbb{R}),\quad \alpha\mapsto [\omega_0]^{2-k}\wedge \alpha,
 \]
  is an isomorphism for all $k\leqslant 2$.
  We choose
  $$
  g_0=\sum_i dx^i\otimes dx^i,\,\,\, \omega_0=dx^1\wedge dx^2+dx^3\wedge dx^4,
  $$
  so $J_0$ is given by
  $$
  J_0dx^1=dx^2,\,\,\,J_0dx^2=-dx^1,\,\,\,J_0dx^3=dx^4,\,\,\,J_0dx^4=-dx^3.
  $$
  Equivalently, $J_0$ may be given by specifying
  $$
  \Lambda^-_{J_0}=Span\{dx^1\wedge dx^3-dx^2\wedge dx^4,dx^1\wedge dx^4+dx^2\wedge dx^3\}.
  $$
  Consider the almost complex structure $J$ given by
  $$
  Jdx^1=mdx^2,\,\,\,Jdx^2=-\frac{1}{m}dx^1,\,\,\,Jdx^3=dx^4,\,\,\,Jdx^4=-dx^3,
  $$
  where $m=m(x^2,x^4)$ is a positive periodic function on $x^2,x^4$ only.
  It ie easy to see that
  $$
  \Lambda^-_J=Span\{dx^1\wedge dx^3-mdx^2\wedge dx^4,dx^1\wedge dx^4+mdx^2\wedge dx^3\},
  $$
  and $J$ is compatible with $\omega_0$.
  $$g(\cdot,\cdot)=\omega_0(\cdot,J\cdot)=\frac{1}{m}dx^1\otimes dx^1+mdx^2\otimes dx^2+dx^3\otimes dx^3+dx^4\otimes dx^4.$$
 We claim that we can choose $m$ such that $h^-_J=1$ (proved by T. Draghici in {\rm \cite{Dra}}).

 Denote by $\psi_1=dx^1\wedge dx^3-mdx^2\wedge dx^4$ and $\psi_2=dx^1\wedge dx^4+mdx^2\wedge dx^3$.
 Note that $m=m(x^2,x^4)$, we have
 \begin{equation}\label{example 1}
   d\psi_1=0,\,\,\, d\psi_2=m_4dx^2\wedge dx^3\wedge dx^4=(\log m)_4dx^4\wedge\psi_2.
 \end{equation}
 Here we denote by $m_4=(\partial m)/(\partial x^4)$ and $(\log m)_4=(\partial(\log m))/(\partial x^4)$.
 The general $J$-anti-invariant form is written as $A\psi_1+B\psi_2$, where $A,B$ are smooth functions on the torus.
 The condition that such a form to be closed is equivalent with
 \begin{equation}\label{example 2}
   dA\wedge\psi_1+(dB+B(\log m)_4dx^4)\wedge\psi_2=0.
 \end{equation}
 We claim that in any solution $(A,B)$ of (\ref{example 2}) $A$ must be a constant.
 To see this, taking the Hodge operator $*_g$ of both sides of the above equation we get
 \begin{equation}\label{example 3}
   JdA=-(dB+B(\log m)_4dx^4).
 \end{equation}
 Taking one more differential,
 and then taking trace with respect to $\omega_0$ we get
 \begin{equation}\label{example 4}
   \Delta A=B_3(\log m)_4,
 \end{equation}
 where $B_3=(\partial B)/(\partial x^3)$.
 By (\ref{example 3}), we get $B_3=A_4:=(\partial A)/(\partial x^4)$.
 Then $(\ref{example 4})$ evolves into
 \begin{equation}\label{example 5}
   -\Delta A+A_4(\log m)_4=0.
 \end{equation}
 By the maximum principle, it follows that $A=const$.
 Plugging this back in (\ref{example 3}), we see that $B$ must satisfy
 $$dB=-B(\log m)_4dx^4.$$
 It is easy to see that the compatibility relation of this system is
 $$B(\log m)_{42}=0.$$
 Thus, if $(\log m)_{42}\neq 0$ (Indeed, it will be sufficient that the locus of $(\log m)_{42}$ has zero Lebesgue measure.),
 the only solutions for (\ref{example 2}) are $A=const$
 and $B=0$, hence $h^-_J=1$.
 We choose $m=e^{\sin 2\pi(x^2+x^4)}$ to be a positive periodic function on $\mathbb{T}^4$ such that $(\log m)_{42}=-4\pi^2\sin2\pi(x^2+x^4)$
 has zero Lebesgue measure.
 So on $(\mathbb{T}^4, g, J,\omega_0)$, $h^-_J$ is equal to $1$.
 Of course, $J$ is $C^\infty$ pure and full
 (T. Draghici, T.-J. Li and W. Zhang proved in \cite{DLZS} that on any closed almost complex 4-manifold $(M,J)$,
 $J$ is $C^\infty$ pure and full.).
 Additionally, Lejmi \cite{LejS1} proved that, on a closed almost K\"{a}hler $4$-manifold $(M,g,J,\omega)$,
 the kernel of $P_J$ consists of primitive harmonic 2-forms and $\dim\ker P_J=b^-+h^-_J$.
 So on $(\mathbb{T}^4, g, J,\omega_0)$, $\dim\ker P_J=b^-+h^-_J=4<5=b^2-1$.

 Let us denote by $e_i\triangleq dx^i$ and $e_{ij}\triangleq dx^i\wedge dx^j$. Please see the following table.

  \begin{center}
  \begin{tabular}{|c|l|}
    \hline

   $\mathcal Z_{J_0}^+$ &  $\omega_0$, \,\,\, $e_{12}-e_{34}$, \,\,\, $e_{13}+e_{24}$, \,\,\, $e_{14}-e_{23}$  \\

   \hline

   $\mathcal Z_{J_0}^-$ &  $e_{13}-e_{24}$, \,\,\, $e_{14}+e_{23}$    \\

   \hline

   $\ker P_{J_0}$ & $e_{12}-e_{34}$, \,\,\, $e_{13}+e_{24}$, \,\,\,$e_{14}-e_{23}$, \,\,\, $e_{13}-e_{24}$, \,\,\, $e_{14}+e_{23}$   \\

   \hline

    $\mathcal{H}^+_{g_0}$ &  $\omega_0$, \,\,\, $e_{13}-e_{24}$, \,\,\, $e_{14}+e_{23}$ \\

   \hline

    $\mathcal{H}^-_{g_0}$ &  $e_{12}-e_{34}$, \,\,\, $e_{13}+e_{24}$, \,\,\, $e_{14}-e_{23}$  \\

   \hline

   $\mathcal Z^+_J$ &  $\omega_0$, \,\,\, $e_{12}-e_{34}$, \,\,\, $e_{13}+me_{24}$,  \,\,\, $\frac{1}{1+m}(e_{12}-e_{34}+e_{14}-me_{23})$,  \\
                 &              $\frac{1}{1-m}(e_{12}+e_{34}+e_{14}-me_{23})$            \\
   \hline
   $\mathcal Z^-_J$ &  $e_{13}-me_{24}$              \\
   \hline
   $\ker P_J$ &  $e_{13}-me_{24}$, \,\,\, $e_{12}-e_{34}$, \,\,\, $e_{13}+me_{24}$,  \,\,\, $\frac{1}{1+m}(e_{12}-e_{34}+e_{14}-me_{23})$   \\
   \hline
   $\mathcal{H}^+_g$ &  $\omega_0$, \,\,\, $e_{13}-me_{24}$,\,\,\, $\frac{1}{1+m}(\omega_0+e_{14}+me_{23})$   \\
   \hline
   $\mathcal{H}^-_g$ &  $e_{12}-e_{34}$, \,\,\, $e_{13}+me_{24}$,\,\,\, $\frac{1}{1+m}(e_{12}-e_{34}+e_{14}-me_{23})$   \\
   \hline
  \end{tabular}
  \end{center}
  \begin{center}
  {\bf Table 1.} Bases for $\mathcal Z_{J_0}^+$, $\mathcal Z_{J_0}^-$, $\ker P_{J_0}$, etc.  of $\mathbb{T}^4$.
  \end{center}
   }
  \end{exa}

 \vskip 6pt

 By the above example,
 we want to propose the following question:
 \begin{ques}
 On any closed symplectic $4$-manifold $(M,\omega)$, is there a $\omega$-compatible (or $\omega$-tame) almost complex structure $J$
 such that $\dim\ker P_J=b^2-1$ ?
 \end{ques}

 \begin{rem}
 On any closed symplectic $4$-manifold $(M,\omega)$, if an almost complex structure $J$ is tamed by symplectic form $\omega$ and $h^-_J=b^+-1$,
 then it implies that there exists the generalized $\partial\overline{\partial}$-lemma (cf. {\rm \cite{LejS2,TWZ} }).

 \end{rem}

  The above example $\mathbb{T}^4$ admits a K\"{a}hler structure $(g_0,J_0,\omega_0)$.
  In Section 3, we will give a non-K\"{a}hler example
  $M^6(c)$ which is constructed by M. Fern\'{a}ndez, V. Mu\~{n}oz and J. A. Santisteban.
  They have proven that $M^6(c)$ does not admit any K\"{a}hler metric (cf. \cite{FMS}).
  We will prove that there exists an almost complex structure $J$ on $M^6(c)$ such that $\dim\ker P_J=b^2-1$.
  Please see the Example \ref{example2}.

 \section{Primitive symplectic cohomology of degree two}\setcounter{equation}{0}

   L.-S. Tseng and S.-T. Yau \cite{TY} considered new symplectic cohomology groups
  \begin{eqnarray*}
    H^k_{d+d^\Lambda}(M)&=& \frac{{\rm Ker}(d+d^\Lambda)\cap\Omega^k(M)}{{\rm Im}(dd^\Lambda)\cap\Omega^k(M)}.
  \end{eqnarray*}
  and
  \begin{eqnarray*}
   H^k_{dd^\Lambda}(M)&=&  \frac{{\rm Ker}(dd^\Lambda)\cap\Omega^k(M)}{({\rm Im}\,d+{\rm Im}\,d^\Lambda)\cap\Omega^k(M)}
  \end{eqnarray*}
  on a compact symplectic manifold $(M,\omega)$ of dimension $2n$.
  We denote the spaces of $d+d^\Lambda$ harmonic $k$-forms and $dd^\Lambda$ harmonic $k$-forms by
  $\mathcal{H}^k_{d+d^\Lambda}(M)$ and $\mathcal{H}^k_{dd^\Lambda}(M)$, respectively.
  For any almost K\"{a}hler triple $(g,J,\omega)$,
  a $k$-form $\alpha\in\Omega^k(M)$ is said to be $d+d^\Lambda$-harmonic (see \cite{TY}) if
  \begin{eqnarray*}
  d\alpha=d^\Lambda\alpha=0 \quad \text{ and } \quad (dd^\Lambda)^*\alpha=0,
  \end{eqnarray*}
  and $dd^\Lambda$-harmonic (see \cite{TY}) if
  \begin{eqnarray*}
  d^*\alpha=(d^\Lambda)^*\alpha=0 \quad \text{ and } \quad dd^\Lambda\alpha=0,
  \end{eqnarray*}
  where $d^*=-*_gd*_g$, $d^{\Lambda*}=*_gd^\Lambda*_g$ and $(dd^\Lambda)^*=(-1)^{k+1}*_gdd^\Lambda*_g$.
  Tseng and Yau also proved that $\mathcal{H}^k_{d+d^\Lambda}(M)$ and
  $\mathcal{H}^k_{dd^\Lambda}(M)$ are finite dimensional and
  isomorphic to $H^k_{d+d^\Lambda}(M)$ and $H^k_{dd^\Lambda}(M)$, respectively.
  Let $\mathcal{H}^-_{d+d^\Lambda}(M)$ and $\mathcal{H}^-_{dd^\Lambda}(M)$ denote the spaces of primitive
  $d+d^\Lambda$ harmonic $2$-forms and primitive $dd^\Lambda$ harmonic $2$-forms, respectively.
  \begin{eqnarray*}
    \mathcal{H}^-_{d+d^\Lambda}(M)\triangleq\mathcal{H}^2_{d+d^\Lambda}(M)\cap\Omega^2_0, \,\,\,
    \mathcal{H}^-_{dd^\Lambda}(M)\triangleq\mathcal{H}^2_{dd^\Lambda}(M)\cap\Omega^2_0 .
  \end{eqnarray*}
  Hence, we can get the following decompositions
  \begin{eqnarray*}
    \mathcal{H}^2_{d+d^\Lambda}(M)=Span_{\mathbb{R}}\{\omega\}\oplus\mathcal{H}^-_{d+d^\Lambda}(M),
  \end{eqnarray*}
  \begin{eqnarray*}
    \mathcal{H}^2_{dd^\Lambda}(M)=Span_{\mathbb{R}}\{\omega\}\oplus\mathcal{H}^-_{dd^\Lambda}(M).
  \end{eqnarray*}

  \begin{defi}
  Let $(M,g,J,\omega)$ be a closed almost K\"{a}hler $4$-manifold.
  Set
  $$(\mathcal{H}^-_J\oplus\mathcal{H}_g^-)^{-,\perp}_{d+d^\Lambda}=\{\alpha\in\mathcal{H}^-_{d+d^\Lambda}\,\,\,|\,\,\,\alpha=d_J^-\theta^1+d^-_g\theta^2\}$$
  and $$(\mathcal{H}^-_J\oplus\mathcal{H}_g^-)^{-,\perp}_{dd^\Lambda}=\{\alpha\in\mathcal{H}^-_{dd^\Lambda}\,\,\,|\,\,\,\alpha=d_J^-\theta^1+d^-_g\theta^2\},$$
  where $d_J^-=P_J^-\circ d$, $d_g^-=P_g^-\circ d$ and $\theta^1,\theta^2\in\Omega^1(M)$.
  Here $P^-_J$ is the projection from $\Omega^2(M)$ to $\Omega^-_J(M)$ and $P^-_g$ is the projection from $\Omega^2(M)$ to $\Omega^-_g(M)$.
  \end{defi}

  \begin{theo}
   Suppose that $(M,g,J,\omega)$ is a closed almost K\"{a}hler $2n$-manifold, then
   $$\ker P_J=\mathcal{H}^-_{d+d^\Lambda}(M)\cap \mathcal{H}^-_{dd^\Lambda}(M).$$
   If $\mathcal{H}^-_{d+d^\Lambda}(M)=\mathcal{H}^-_{dd^\Lambda}(M)$, then
   $$\mathcal{H}^2_{d+d^\Lambda}(M)=\mathcal{H}^2_{dd^\Lambda}(M)=Span_{\mathbb{R}}\{\omega\}\oplus \mathcal{H}^-_J\oplus\mathcal{H}^+_{J,0}.$$
   In particular, if $n=2$, $$\mathcal{H}^-_{d+d^\Lambda}(M)=\mathcal{H}^-_J\oplus\mathcal{H}_g^-\oplus(\mathcal{H}^-_J\oplus\mathcal{H}_g^-)^{-,\perp}_{d+d^\Lambda},$$
   $$\mathcal{H}^-_{dd^\Lambda}(M)=\mathcal{H}^-_J\oplus\mathcal{H}_g^-\oplus(\mathcal{H}^-_J\oplus\mathcal{H}_g^-)^{-,\perp}_{dd^\Lambda},$$
   $$*_g(\mathcal{H}^-_J\oplus\mathcal{H}_g^-)^{-,\perp}_{d+d^\Lambda}=(\mathcal{H}^-_J\oplus\mathcal{H}_g^-)^{-,\perp}_{dd^\Lambda}.$$
  \end{theo}
  \begin{proof}
   Let us begin with the first assertion of the Theorem.
   We claim that $\mathcal{H}^-_J\oplus\mathcal{H}^+_{J,0}\subset\mathcal{H}^-_{d+d^\Lambda}(M)$
   and $\mathcal{H}^-_J\oplus\mathcal{H}^+_{J,0}\subset\mathcal{H}^-_{dd^\Lambda}(M)$.
   Indeed, for any $\alpha=\beta+\gamma\in\mathcal{H}^-_J\oplus\mathcal{H}^+_{J,0}$,
   $\beta\in\mathcal{H}^-_J$, $\gamma\in\mathcal{H}^+_{J,0}$.
   By Weil's identity, we have
   $*_s\alpha=-\frac{L^{n-2}}{(n-2)!}\alpha$.
   Then
   $$
   d*_s\alpha=-d\frac{L^{n-2}}{(n-2)!}\alpha=-\frac{L^{n-2}}{(n-2)!}d\alpha=0.
   $$
   Hence, $d^\Lambda\alpha=0$.
   Also by Weil's identity, we have
   $$
   *_s*_g\alpha=*_s*_g\beta+*_s*_g\gamma=*_s\frac{L^{n-2}}{(n-2)!}\beta-*_s\frac{L^{n-2}}{(n-2)!}\gamma=-\beta+\gamma.
   $$
   Hence, $d*_s*_g\alpha=d(-\beta+\gamma)=0$.
   So we get $(dd^\Lambda)^*\alpha=0$.
   Therefore, $\alpha\in\mathcal{H}^-_{d+d^\Lambda}(M)$ and $\mathcal{H}^-_J\oplus\mathcal{H}^+_{J,0}\subset\mathcal{H}^-_{d+d^\Lambda}(M)$.
   It is similar for $\mathcal{H}^-_J\oplus\mathcal{H}^+_{J,0}\subset\mathcal{H}^-_{dd^\Lambda}(M)$.
   In particular, if $n=2$, $\mathcal{H}^+_{J,0}=\mathcal{H}^-_g$, we can get
   $\mathcal{H}^-_J\oplus\mathcal{H}^-_g\subset\mathcal{H}^-_{d+d^\Lambda}(M)$
   and $\mathcal{H}^-_J\oplus\mathcal{H}^-_g\subset\mathcal{H}^-_{dd^\Lambda}(M)$.
   So we can get
   $$\mathcal{H}^-_J\oplus\mathcal{H}^+_{J,0}\subset\mathcal{H}^-_{d+d^\Lambda}(M)\cap\mathcal{H}^-_{dd^\Lambda}(M).$$
   For the other hand, it follows straightforwardly from the definitions of
   $\mathcal{H}^-_{d+d^\Lambda}(M)$, $\mathcal{H}^-_{dd^\Lambda}(M)$ and $\ker P_J$.
   If $\mathcal{H}^-_{d+d^\Lambda}(M)=\mathcal{H}^-_{dd^\Lambda}(M)$
   (i.e. $\mathcal{H}^2_{d+d^\Lambda}(M)=\mathcal{H}^2_{dd^\Lambda}(M)$), then $\ker P_J=\mathcal{H}^-_{d+d^\Lambda}(M)=\mathcal{H}^-_{dd^\Lambda}(M)$.
   Hence, $$\mathcal{H}^2_{d+d^\Lambda}(M)=\mathcal{H}^2_{dd^\Lambda}(M)=Span_{\mathbb{R}}\{\omega\}\oplus \ker P_J=Span_{\mathbb{R}}\{\omega\}\oplus \mathcal{H}^-_J\oplus\mathcal{H}^+_{J,0}.$$

   In the following, we suppose that $(M,g,J,\omega)$ is a closed almost K\"{a}hler $4$-manifold.
   Then it is easy to see that $\Omega^-_g=\Omega^+_{J,0}$.
   So $\Omega^2$ can be written as
   $$\Omega^2=\Omega^2_1\oplus\Omega^2_0=\Omega^2_1\oplus\Omega^-_g\oplus\Omega^-_J.$$
   For any $\alpha\in\mathcal{H}^2_{d+d^\Lambda}$,
   by the definition of $\mathcal{H}^2_{d+d^\Lambda}$,
   $$
   d\alpha=d^\Lambda\alpha=0,\,\,\,(dd^\Lambda)^*\alpha=0,
   $$
   where $(dd^\Lambda)^*=-*_gdd^\Lambda*_g$.
   It is clear that
  \begin{equation*}
  d\alpha=0, \,\,\,d*_s\alpha=0.
  \end{equation*}
  Hence
  $$
  d\frac{1}{2}(1+*_s)\alpha=0.
  $$
  Since $\frac{1}{2}(1+*_s)\alpha\in \Omega^2_1$, it can be written as
  $$
  \frac{1}{2}(1+*_s)\alpha=f_\alpha\omega.
  $$
  Since $d\omega=0$, we have $d(f_\alpha\omega)=df_\alpha\wedge\omega=0$.
  It follows that $f_\alpha=c_\alpha$ is a constant since $\omega$ is nondegenerate.

  Let
  $$
  \alpha_1=\alpha-c_\alpha\omega=\frac{1}{2}(1-*_s)\alpha\in\Omega^2_0=\Omega^-_J\oplus\Omega^-_g.
  $$
  Hence, $\alpha_1$ is still in $\mathcal{H}^2_{d+d^\Lambda}$ and $\alpha_1$ can be written as
  $$
  \alpha_1=\alpha_{1,J}^-+\alpha_{1,g}^-,
  $$
  where $\alpha_{1,J}^-\in\Omega^-_J$ and $\alpha_{1,g}^-\in\Omega^-_g$.
  By Lejmi lemma \cite{LejS1} and Hodge decomposition \cite{DKT},
  $$
  \alpha_{1,J}^-=\beta_\alpha+d_J^-d^*\eta_\alpha=\beta_\alpha+d_J^-\theta^1_\alpha,\,\,\,\theta^1_\alpha=d^*\eta_\alpha
  $$
  and
  $$
  \alpha_{1,g}^-=\gamma_\alpha+d_g^-d^*\xi_\alpha=\gamma_\alpha+d_g^-\theta^2_\alpha,\,\,\,\theta^2_\alpha=d^*\xi_\alpha,
  $$
  where $\beta_\alpha\in \mathcal{Z}_J^-=\mathcal{H}_J^-$, $\gamma_\alpha\in \mathcal{H}^-_g$,
  $\eta_\alpha\in\Omega^-_J$ and $\xi_\alpha\in\Omega^-_g$.
  Then $$\alpha=c_\alpha\omega+\alpha_1=c_\alpha\omega+\beta_\alpha+\gamma_\alpha+(d_J^-\theta^1_\alpha+d_g^-\theta^2_\alpha)$$
  and $(d_J^-\theta^1_\alpha+d_g^-\theta^2_\alpha)\in(\mathcal{H}^-_J\oplus\mathcal{H}_g^-)^{-,\perp}_{d+d^\Lambda}$.
  So we can get that
  $$
  \mathcal{H}^2_{d+d^\Lambda}=Span_{\mathbb{R}}\{\omega\}\oplus \mathcal{H}_J^-\oplus \mathcal{H}^-_g\oplus(\mathcal{H}^-_J\oplus \mathcal{H}^-_g)^{-,\perp}_{d+d^\Lambda}.
  $$
  It is easy to see that
  $$
 (\mathcal{H}^-_J\oplus\mathcal{H}_g^-)^{-,\perp}_{d+d^\Lambda}\subset\mathcal{H}^-_{d+d^\Lambda}
  $$
  is just the orthogonal complement of $\mathcal{H}_J^-\oplus\mathcal{H}^-_g$ in $\mathcal{H}^-_{d+d^\Lambda}$ with respect to the cup product.
  Similarly, one has
  $$
  \mathcal{H}^2_{dd^\Lambda}=Span_{\mathbb{R}}\{\omega\}\oplus \mathcal{H}_J^-\oplus \mathcal{H}^-_g\oplus(\mathcal{H}^-_J\oplus \mathcal{H}^-_g)^{-,\perp}_{dd^\Lambda},
  $$
  since $*_g\mathcal{H}^2_{d+d^\Lambda}=\mathcal{H}^2_{dd^\Lambda}$
  and $*_g\mathcal{H}^-_{d+d^\Lambda}=\mathcal{H}^-_{dd^\Lambda}$ (see \cite[Proposition 3.24]{TY}).
  This completes the proof of the Theorem.
  \end{proof}

  It is helpful to have explicit examples showing clearly the differences between the different cohomologies and $\ker P_J$
  discussed above. For this we consider the following examples.
  \begin{exa}\label{example2}
  {\rm (cf. \cite{FMS})
  Let $G(c)$ be the connected completely solvable Lie group of dimension $5$ consisting of matrices of the form
  \begin{equation}
     a=\left(
     \begin{array}{cccccc}
       e^{cz} & 0 & 0 & 0 & 0 & x_1 \\
       0 & e^{cz} & 0 & 0 & 0 & y_1 \\
       0 & 0 & e^{cz} & 0 & 0 & x_2 \\
       0 & 0 & 0 & e^{cz} & 0 & y_2 \\
       0 & 0 & 0 & 0 & 1 & z \\
       0 & 0 & 0 & 0 & 0 & 1 \\
     \end{array}
    \right),
   \end{equation}
   where $x_i,\,\,y_i,\,\,z\in\mathbb{R}$ $(i=1,2)$ and $c$ is a nonzero real number.
   Then a global system of coordinates $x_1,\,\,y_1,\,\,x_2,\,\,y_2$ and $z$ for $G(c)$ is given by
   $x_i(a)=x_i$, $y_i(a)=y_i$ and $z(a)=z$.
   A standard calculation shows that a basis for the right invariant $1$-forms on $G(c)$ consists of
   \begin{equation}
     \{dx_1-cx_1dz,\,\,dy_1-cy_1dz,\,\,dx_2-cx_2dz,\,\,dy_2-cy_2dz,\,\,dz\}.
   \end{equation}
   Alternatively, the Lie group $G(c)$ may be described as a semidirect product $G(c)=\mathbb{R}\ltimes_{\psi}\mathbb{R}^4$,
   where $\psi(z)$ is the linear transformation of $\mathbb{R}^4$ given by the matrix
   \begin{equation}
    \left(
      \begin{array}{cccc}
        e^{cz} & 0 & 0 & 0 \\
        0 & e^{-cz}& 0 & 0 \\
        0 & 0 & e^{cz} & 0 \\
        0 & 0 & 0 & e^{-cz} \\
      \end{array}
    \right),
   \end{equation}
   for any $z\in\mathbb{R}$.
   Thus, $G(c)$ has a discrete subgroup $\Gamma(c)=\mathbb{Z}\ltimes_{\psi}\mathbb{Z}^4$ such that the quotient
   space $G(c)/\Gamma(c)$ is compact. Therefore, the forms
   $dx_i-cx_idz$, $dy_i-cy_idz$ and $dz$ $(i=1,2)$ descend to $1$-forms $\alpha_i$, $\beta_i$ and $\gamma$ $(i=1,2)$
   on $G(c)/\Gamma(c)$.

   M. Fern\'{a}ndez, V. Mu\~{n}oz and J. A. Santisteban considered the manifold $M^6(c)=G(c)/\Gamma(c)\times S^1$.
   Here, there are $1$-forms $\alpha_1$, $\beta_1$, $\alpha_2$, $\beta_2$, $\gamma$ and $\eta$ on $M^6(c)$ such that
   \begin{equation}
   d\alpha_i=-c\alpha_i\wedge\gamma,\,\,\,d\beta_i=-c\beta_i\wedge\gamma,\,\,\,d\gamma=d\eta=0,
   \end{equation}
   where $i=1,2$ and such that at each point of $M^6(c)$, $\{\alpha_1, \beta_1, \alpha_2, \beta_2, \gamma,\eta\}$
   is a basis for the $1$-forms on $M^6(c)$.
   Using Hattori's theorem, they compute the real cohomology of $M^6(c)$:
   \begin{eqnarray}
    H^0(M^6(c)) &=& \langle 1\rangle, \nonumber\\
     H^1(M^6(c)) &=& \langle [\gamma],\,\,[\eta]\rangle, \nonumber\\
     H^2(M^6(c)) &=& \langle    [\alpha_1\wedge\beta_1],\,\,[\alpha_1\wedge\beta_2],\,\,[\alpha_2\wedge\beta_1],\,\,[\alpha_2\wedge\beta_2],\,\,[\gamma\wedge\eta]\rangle, \nonumber\\
     H^3(M^6(c)) &=& \langle[\alpha_1\wedge\beta_1\wedge\gamma],\,\,[\alpha_1\wedge\beta_2\wedge\gamma],\,\,[\alpha_2\wedge\beta_1\wedge\gamma],\,\,
                     [\alpha_2\wedge\beta_2\wedge\gamma], \nonumber\\
                 && [\alpha_1\wedge\beta_1\wedge\eta],\,\,[\alpha_1\wedge\beta_2\wedge\eta],\,\,[\alpha_2\wedge\beta_1\wedge\eta],\,\,
                     [\alpha_2\wedge\beta_2\wedge\eta]\rangle, \nonumber\\
     H^4(M^6(c)) &=& \langle [\alpha_1\wedge\beta_1\wedge\alpha_2\wedge\beta_2],\,\,[\alpha_1\wedge\beta_1\wedge\gamma\wedge\eta],\,\,
                       [\alpha_1\wedge\beta_2\wedge\gamma\wedge\eta], \nonumber\\
                 && [\alpha_2\wedge\beta_1\wedge\gamma\wedge\eta],\,\,[\alpha_2\wedge\beta_2\wedge\gamma\wedge\eta] \rangle,\nonumber\\
     H^5(M^6(c))&=& \langle [\alpha_1\wedge\beta_1\wedge\alpha_2\wedge\beta_2\wedge\gamma],\,\,
     [\alpha_1\wedge\beta_1\wedge\alpha_2\wedge\beta_2\wedge\eta]\rangle, \nonumber\\
     H^6(M^6(c))&=& \langle [\alpha_1\wedge\beta_1\wedge\alpha_2\wedge\beta_2\wedge\gamma\wedge\eta]\rangle.
   \end{eqnarray}
   Therefore, the Betti number of $M^6(c)$ are
   \begin{eqnarray}
     b^0 &=& b^6=1,\nonumber\\
     b^1 &=& b^5=2,\nonumber \\
     b^2 &=& b^4=5, \nonumber\\
     b^3 &=& 8 .
   \end{eqnarray}
   We denote by $(g,J,\omega)$ be an almost K\"{a}hler structure on $M^6(c)$, where we choose
   $$g=\alpha_1\otimes\alpha_1+\beta_1\otimes\beta_1+\alpha_2\otimes\alpha_2+\beta_2\otimes\beta_2+\gamma\otimes\gamma+\eta\otimes\eta$$
   and
   $$\omega=\alpha_1\wedge\beta_1+\alpha_2\wedge\beta_2+\gamma\wedge\eta.$$
   So $J$ is given by
   $$
   J\alpha_1=\beta_1,\,\,\,J\alpha_2=\beta_2,\,\,\,J\gamma=\eta.
   $$
   It is clear that the maps
   $$[\omega]:H^2_{dR}(M^6(c);\mathbb{R})\rightarrow H^4_{dR}(M^6(c);\mathbb{R})$$
   and
   $$[\omega]^2:H^1_{dR}(M^6(c);\mathbb{R})\rightarrow H^5_{dR}(M^6(c);\mathbb{R})$$
   are isomorphisms.
   Thus, $(M^6(c),\omega)$ satisfies the hard Lefschetz property.
   By simple calculation, we can get
   \begin{eqnarray}
     \mathcal{Z}^-_J &=& Span_{\mathbb{R}}\{\alpha_1\wedge\beta_2-\alpha_2\wedge\beta_1\},  \\
    \mathcal{Z}^+_J &=& Span_{\mathbb{R}}\{\alpha_1\wedge\beta_2+\alpha_2\wedge\beta_1,\,\,\alpha_1\wedge\beta_1,\,\,
                  \alpha_2\wedge\beta_2,\,\,\gamma\wedge\eta\},  \\
     \ker P_J&=& Span_{\mathbb{R}}\{\alpha_1\wedge\beta_2-\alpha_2\wedge\beta_1,\,\,\alpha_1\wedge\beta_2+\alpha_2\wedge\beta_1,  \nonumber \\
             &&  \alpha_1\wedge\beta_1-\gamma\wedge\eta,\,\,\alpha_2\wedge\beta_2-\gamma\wedge\eta   \}.
   \end{eqnarray}
   Hence, $\dim\ker P_J=4=b^2-1$. Of course, $J$ is $C^\infty$ pure and full.

   \begin{center}
  \begin{tabular}{|c|l|}
    \hline
   $H^2_{dR}$ &  $\alpha_1\wedge\beta_1,\,\,\alpha_1\wedge\beta_2,\,\,\alpha_2\wedge\beta_1,\,\,\alpha_2\wedge\beta_2,\,\,\gamma\wedge\eta$       \\
    \hline
   $\mathcal{Z}^+_J$ &  $\alpha_1\wedge\beta_2+\alpha_2\wedge\beta_1,\,\,\alpha_1\wedge\beta_1,\,\, \alpha_2\wedge\beta_2,\,\,\gamma\wedge\eta$       \\
   \hline
   $\mathcal{Z}^-_J$ &  $\alpha_1\wedge\beta_2-\alpha_2\wedge\beta_1$              \\
   \hline
   $\mathcal{H}^-_{d+d^\Lambda}$ & $\alpha_1\wedge\beta_2-\alpha_2\wedge\beta_1,\,\, \alpha_1\wedge\beta_2+\alpha_2\wedge\beta_1,
                         \,\, \alpha_1\wedge\beta_1-\gamma\wedge\eta,\,\, \alpha_2\wedge\beta_2-\gamma\wedge\eta$     \\
   \hline
   $\mathcal{H}^-_{dd^\Lambda}$ & $\alpha_1\wedge\beta_2-\alpha_2\wedge\beta_1,\,\, \alpha_1\wedge\beta_2+\alpha_2\wedge\beta_1,
                         \,\, \alpha_1\wedge\beta_1-\gamma\wedge\eta,\,\, \alpha_2\wedge\beta_2-\gamma\wedge\eta$     \\
   \hline
   $\ker P_J$ &   $\alpha_1\wedge\beta_2-\alpha_2\wedge\beta_1,\,\,\alpha_1\wedge\beta_2+\alpha_2\wedge\beta_1,
                         \alpha_1\wedge\beta_1-\gamma\wedge\eta,\,\,\alpha_2\wedge\beta_2-\gamma\wedge\eta $      \\
   \hline
  \end{tabular}
  \end{center}
  \begin{center}
  {\bf Table 2.} Bases for $H^2_{dR}$, $\mathcal{Z}^+_J$, $\mathcal{Z}^-_J$, $\mathcal{H}^-_{d+d^\Lambda}$, $\mathcal{H}^-_{dd^\Lambda}$ and $\ker P_J$ of $M^6(c)$.
  \end{center}
  By the above table, we can see that $\mathcal{H}^-_{d+d^\Lambda}(M^6(c))=\mathcal{H}^-_{dd^\Lambda}(M^6(c))$.
  So
  $$\mathcal{H}^2_{d+d^\Lambda}(M^6(c))=\mathcal{H}^2_{dd^\Lambda}(M^6(c))=Span_{\mathbb{R}}\{\omega\}\oplus \mathcal{H}^-_J\oplus\mathcal{H}^+_{J,0}.$$
   }
   \end{exa}

  \begin{prop}
  (cf. \cite{FMS})
  The manifold $M^6(c)$ does not admit K\"{a}hler metric.
  \end{prop}

  \begin{exa}
  {\rm (cf. \cite{TY})
  Let $M$ be the Kodaira-Thurston nilmanifold defined by taking $\mathbb{R}^4$ and modding out by the identification
   $$(x_1,x_2,x_3,x_4)\sim (x_1+a,x_2+b,x_3+c,x_4+d-bx_3),$$
   where $a,b,c,d\in \mathbb{Z}$.
   The resulting manifold is a torus bundle over a torus with a basis of cotangent $1$-forms given by
   $$e_1=dx_1,\,\,\,e_2=dx_2,\,\,\,e_3=dx_3,\,\,\,e_4=dx_4+x_2dx_3.$$
   It is well known that Kodaira-Thurston manifold admits no K\"{a}hler structure.
   We take the symplectic form to be
   $$\omega=e_1\wedge e_2+e_3\wedge e_4.$$
   Consider the $\omega$-compatible almost complex structure $J$ given by
   $$J(e_1)=e_2,\,\,\,J(e_2)=-e_1,\,\,\,J(e_3)=e_4,\,\,\,J(e_4)=-e_3.$$
   Let us denote the compatible metric by
   $$g(\cdot,\cdot)=\omega(\cdot,J\cdot)=e_1\otimes e_1+e_2\otimes e_2+e_3\otimes e_3+e_4\otimes e_4.$$
   $(g,J,\omega)$ is an almost K\"{a}hler structure but not K\"{a}hler since the almost complex structure $J$ is not integrable.
   Please see the following table for the relationship between $H^2_{dR}$, $H^2_{d+d^\Lambda}$, $H^2_{dd^\Lambda}$ and $\ker P_J$.
   \begin{center}
  \begin{tabular}{|c|l|}
   \hline

   $H^2_{dR}$ & $\omega$, \,\,\, $e_1\wedge e_2-e_3\wedge e_4$, \,\,\, $e_1\wedge e_3$,  \,\,\, $e_2\wedge e_4$   \\

   \hline

    $\mathcal{Z}^+_J$ & $\omega$, \,\,\, $e_1\wedge e_2-e_3\wedge e_4$    \\

   \hline

   $\mathcal{Z}^-_J$ &  $e_1\wedge e_3$,  \,\,\, $e_2\wedge e_4$   \\

   \hline

   $H^2_{d+d^\Lambda}$ & $\omega$, \,\,\, $e_1\wedge e_2-e_3\wedge e_4$, \,\,\, $e_1\wedge e_3$, \,\,\, $e_2\wedge e_4$, \,\,\, $e_2\wedge e_3$  \\

   \hline
   $H^2_{dd^\Lambda}$ & $\omega$, \,\,\, $e_1\wedge e_2-e_3\wedge e_4$, \,\,\, $e_1\wedge e_3$,  \,\,\, $e_2\wedge e_4$, \,\,\, $e_1\wedge e_4$ \\

   \hline
   $\ker P_J$ &   $e_1\wedge e_2-e_3\wedge e_4$, \,\,\, $e_1\wedge e_3$,  \,\,\, $e_2\wedge e_4$     \\
   \hline
  \end{tabular}
  \end{center}
  \begin{center}
  {\bf Table 3.} Bases for $H^2_{dR}$, $\mathcal{Z}^+_J$, $\mathcal{Z}^-_J$, $H^2_{d+d^\Lambda}$, $H^2_{dd^\Lambda}$ and $\ker P_J$ of $M$.
  \end{center}
  By the above table, we can see that $(M,\omega)$ does not satisfy the hard Lefschetz property.
  It is easy to see that the dimension of $\ker P_J$ is equal to $b^2-1=3$.
  }
  \end{exa}

 \vskip 12pt

 \noindent{\bf Acknowledgements.}\,
 The authors would like to thank T. Draghici for stimulating email discussions and providing example.
 The second author would like to thank Fudan University and Jiaxing Hong
 for hosting his visit in the fall semester in 2013.

  \vskip 24pt

  \noindent Qiang Tan\\
  School of Mathematical Sciences, Yangzhou University, Yangzhou, Jiangsu 225002, China\\
  e-mail: tanqiang1986@hotmail.com\\

  \vskip 6pt

  \noindent Hongyu Wang\\
  School of Mathematical Sciences, Yangzhou University, Yangzhou, Jiangsu 225002, China\\
  e-mail: hywang@yzu.edu.cn\\

  \vskip 6pt

  \noindent Jiuru Zhou\\
  School of Mathematical Sciences, Yangzhou University, Yangzhou, Jiangsu 225002, China\\
  e-mail: zhoujr1982@hotmail.com


\begin{thebibliography}{99}
  \bibliographystyle{siam}
  \bibitem{Angella} D. Angella and A. Tomassini, {\it Symplectic manifolds and cohomological decomposition},
                  arXiv:1211.2565v1, [math.SG].

  \bibitem{Angella2} D. Angella and A. Tomassini,  {\it On cohomological decomposition of almost-complex manifolds and deformations},
                       J. Symp. Geom., {\bf 3} (2011), 403--428.

  \bibitem{BesE} A. L. Besse, {\it Einstein manifolds}, Ergeb. Math. Grenzgeb., Springer-Verlag, Berlin, Heidelbeg, New York, 1987.

  \bibitem{BrA} J. L. Brylinski, {\it A differential complex for Poisson manifolds},
                J. Differential Geom., {\bf 28} (1988), 93--114.

  \bibitem{Barth} W. Barth, K. Hulek, C. Peters and A. Van de Ven,
                 {\it Compact Complex Surfaces}, Springer-Verlag, Berlin, 2004.

  \bibitem{DonT} S. K. Donaldson, {\it Two forms on four manifolds and elliptic equations}, Nankai Tracts Math., {\bf 11},
                 Inspired by S. S. Chern, 153-172, World Sci. Publ., Hackensack, N.J., 2006.

  \bibitem{DKT} S. K. Donaldson and P. B. Kronheimer, {\it The geometry of four-manifolds}, Oxford Mathematical Monographs,
                Oxford Science Publications, New York, 1990.

  \bibitem{Dra}  T. Draghici, Private communication, July 2013.

  \bibitem{DLZS} T. Draghici, T.-J. Li and W. Zhang, {\it Symplectic forms and cohomology decomposition of almost complex four-manifolds},
                 Int. Math. Res. Not., 2010, no. 1, 1--17.

  \bibitem{DLZO} T. Draghici, T.-J. Li and W. Zhang, {\it On the $J$-anti-invariant cohomology of almost complex $4$-manifolds},
                 Quaterly J. Math., {\bf 64} (2013), 83--111.


  \bibitem{FMS} M. Fern\'{a}ndez, V. Mu\~{n}oz and J. A. Santisteban, {\it Cohomologically K\"{a}hler manifolds with no K\"{a}hler metrics},
                    International Journal of Mathematics and Mathematical Sciences, {\bf 52} (2003), 3315--3325.

  \bibitem{FTO} A. Fino and A. Tomassini, {\it On some cohomological properties of almost complex manifolds}, J. Geom. Anal., {\bf 20}
                    (2010), 107--131.

  \bibitem{GH} P. A. Griffiths and J. Harries, {\it Principle of Algebraic Geometry}, New York, Wiley, 1978.

  \bibitem{GuiS} V. Guillemin, {\it Symplectic Hodge theory and the $d\delta$-lemma}, preprint, MIT, 2001.

  \bibitem{KMC} K. Kodaira and J. Morrow, {\it Complex Manifolds}, Holt, Rinehart and Winston, 1971.

  \bibitem{LejS1} M. Lejmi, {\it Stability under deformations of extremal almost-K\"{a}hler metrics in dimension 4},
                    Math. Res. Lett., {\bf 17} (2010), 601--612.


  \bibitem{LejS2} M. Lejmi, {\it Stability under deformations of Hermitian-Einstein almost-K\"{a}hler metrics },
                 arXiv:1204.5438v1, [math.DG].

  \bibitem{LiZhC} T.-J. Li and W. Zhang, {\it Comparing tamed and compatible symplectic cones and cohomological
                  properties of almost complex manifolds}, Comm. Anal. Geom., {\bf 17} (2009), 651--684.

  \bibitem{LinS}  Y. Lin, {\it Symplectic Harmonic theory and the Federer-Fleming deformation theorem}, arXiv:1112.2442v3, [math.SG].


  \bibitem{TWZ} Q. Tan, H. Y. Wang and P. Zhu, {\it On tamed almost complex four-manifolds}, work in progress.

  \bibitem{TWZZ} Q. Tan, H. Y. Wang, Y. Zhang and P. Zhu, {\it On cohomology of almost complex 4-manifolds}, to appear in J. Geom. Anal.,
       DOI 10.1007/s12220-014-9477-2.

  \bibitem{TY} L. S. Tseng and S.-T. Yau, {\it Cohomology and Hodge theory on symplectic manifolds: I},
               J. Differential Geom., {\bf 91} (2012), 383--416.

  \bibitem{WeiI} A. Weil, {\it Introduction \`{a} l'\'{E}tude des Vari\'{e}t\'{e} K\"{a}hl\'{e}riennes},
                 Publications de l'Instiut de Math\'{e}matique de l'Universit\'{e} de Nancago VI, Hermann, Paris, 1958.

  \bibitem{YanH} D. Yan, {\it Hodge structure on symplectic manifolds},
               Adv. Math., {\bf 120} (1996), 143--154.

  \end{thebibliography}
 \end{document}